\documentclass{amsart}

\usepackage[utf8]{inputenc}
\usepackage{amsmath,amsfonts,amssymb,amsthm}
\usepackage{hyperref,graphicx,cleveref}
\usepackage{algorithm,algorithmic}

\usepackage[dvipsnames]{xcolor}
\usepackage[shortlabels]{enumitem}

\theoremstyle{plain}
\newtheorem{theorem}{Theorem}
\newtheorem{lemma}[theorem]{Lemma}
\newtheorem{proposition}[theorem]{Proposition}
\newtheorem{corollary}{Corollary}

\theoremstyle{remark}
\newtheorem{remark}{Remark}
\newtheorem*{remark*}{Remark}

\theoremstyle{definition}
\newtheorem*{definition*}{Definition}

\newtheorem{hypothesis}{Hypothesis}
\newtheorem{example}{Example}
\newtheorem*{problem*}{Problem}

\def\C{\mathbb{C}}
\def\N{\mathbb{N}}
\def\R{\mathbb{R}}

\def\ND{\mathcal{N}}

\def\supp{{\rm supp}}
\def\spann{{\rm span}}

\def\L{\mathcal{L}_c}
\renewcommand{\epsilon}{\varepsilon}

\def\Imm{\mathrm{Im}}

\newcommand{\norm}[1]{{\left\Vert {#1}\right\Vert}}

\DeclareMathOperator*{\ran}{ran}

\begin{document}
\title[Infinite-dimensional inverse problems with finite measurements]{Infinite-dimensional inverse problems \\with finite measurements}

\author{Giovanni S. Alberti}
\address{MaLGa Center, Department of Mathematics, University of Genoa, Via Dodecaneso 35, 16146 Genova, Italy.}
\email{giovanni.alberti@unige.it}

\author{Matteo Santacesaria}
\address{MaLGa Center, Department of Mathematics, University of Genoa, Via Dodecaneso 35, 16146 Genova, Italy.}
\email{matteo.santacesaria@unige.it}


\thanks{The authors are members of the Gruppo Nazionale per l’Analisi Matematica,
la Probabilit\`a e le loro Applicazioni (GNAMPA) of the Istituto Nazionale di Alta
Matematica (INdAM). G.S.\ Alberti is partially supported by the UniGE starting grant ``Curiosity''. M.\ Santacesaria is partially supported by a INdAM -- GNAMPA Project 2019. This material is based upon work supported by the Air Force Office of Scientific Research under award number FA8655-20-1-7027. The authors would like to thank Otmar Scherzer for  explaining some useful details of \cite{dehoop2012}.}

\date{\today}

\begin{abstract}
We present a general framework to study uniqueness, stability and reconstruction for infinite-dimensional inverse problems when only a finite-dimensional approximation of the measurements is available. For a large class of inverse problems satisfying Lipschitz stability we show that the same estimate holds even with a finite number of measurements. We also derive a globally convergent reconstruction algorithm based on the Landweber iteration. This theory applies to nonlinear ill-posed problems such as electrical impedance tomography, inverse scattering and quantitative photoacoustic tomography, under the assumption that the unknown belongs to a finite-dimensional subspace.
\end{abstract}

\maketitle

\section{Introduction}
\subsection{General setup and aim of the paper}
The recovery of an unknown physical quantity from indirect measurements is the main goal of an inverse problem. It is often convenient to consider both the unknown quantity $x$ and the measured data $y$ in a continous setting. This allows for studying, for instance, inverse problems modeled by partial differential equations (PDE) \cite{2017-isakov,2017-hasanov} and, more generally, infinite-dimensional inverse problems where the measurement operator, or forward map, is a function between Banach spaces $F\colon X\to Y$ \cite{2011-kirsch}. The inverse problem consists in the reconstruction of $x$ from the knowledge of the measurements $F(x)$. Several possibly nonlinear inverse problems fit into this framework, including electrical impedance tomography (EIT) \cite{calderon1980,eit-1999,2012-mueller-siltanen}, photoacoustic tomography  \cite{Wang1458,2015-kuchment-kunyansky} (and many other hybrid imaging problems \cite{alberti-capdeboscq-2018}), travel time tomography \cite{2019-stefanov}, and inverse scattering \cite{2013-colton-kress}. 

However, in practice we only have access to a finite-dimensional approximation $Q(F(x))$ of the data, for some finite-rank operator $Q\colon Y\to Y$. It is then critical, in view of the applications, to study how the reconstruction depends on this approximation. The aim of this paper is to provide explicit guarantees for exact recovery for general inverse problems with finite measurements. As we show in this work, this issue is strictly related with that of stability, which we now discuss.

\subsection{Stability of inverse problems and previous work.}
The main motivation of this work comes from some nonlinear inverse problems for PDE, such as Calder\'on's problem for EIT, where the unknown is a conductivity distribution that has to be recovered from boundary voltage and current patterns. This is a severely ill-posed problem, where a small error in the data propagates exponentially to the reconstruction \cite{mandache2001}. In other words, the continuous dependence of the conductivity on the boundary data, also known as stability estimate, is of logarithmic type \cite{alessandrini1988}, and this explains the typical low spatial resolution in the reconstruction.

In the case of general nonlinear inverse problems, such as EIT or inverse scattering, it turns out that a possible way to obtain a stronger stability of Lipschitz type is to assume that the unknown belongs to a known finite-dimensional subspace $W$ of the original space $X$. This has been shown for a number of inverse problems for elliptic PDE \cite{alessandrini1996,2005-alessandrini-vessella,bacchelli2006,beretta2008,2011-beretta-francini,beretta2013,bourgeois2013,beretta2014,beretta2014b,beretta2015,gaburro2015,beretta2016b,beretta2017,alessandrini2017,
alessandrini2017u,Harrach2019b,alessandrini2018,beretta2019}. The main drawback of these results is that, even though the unknown is a finite-dimensional object, infinitely many measurements are still needed. For instance, in the case of EIT, an infinite number of boundary current and voltage data are required to recover a conductivity in a known finite-dimensional space. 

A uniqueness result from a finite number of boundary measurements -- as well as a Lipschitz stability estimate and a reconstruction scheme -- was recently obtained by the authors for the Gel'fand-Calder\'on problem for the Schr\"odinger equation and for EIT under rather general finite-dimensionality assumptions on the unknown \cite{alberti2018}. Afterwards, Lipschitz stability from a finite number of measurements was derived by Harrach  for the complete electrode model in EIT \cite{Harrach_2019} (based on the local problem studied in \cite{Lechleiter_2008}), as well as by R\"uland-Sincich for the fractional Calder\'on problem  \cite{ruland2018}. The linearized EIT problem was previously addressed in \cite{alberti2017infinite} using compressed sensing.
Many related works, mainly for the inverse scattering problem, consider a periodic, polygonal or polyhedral structure in the unknown \cite{friedman1989,cheng2003,alessandrini2005b,bao2011,hu2016,blaasten2016,blaasten2017,liu2017}.

Though the above results constitute a clear improvement showing that a finite number of measurements is enough, they still leave many unanswered questions. In our previous result \cite{alberti2018}, the (finite number of) boundary measurements depend on the unknown conductivity, which is clearly impractical. This assumption was then removed in \cite{Harrach_2019}, though the result is not constructive (the number of measurements is not explicitly given), and the approach is restricted to finite-dimensional subspaces $W$ of piecewise analytic conductivities. The other results either work under very restrictive assumptions on the unknown, or they do not provide criteria to choose the measurements depending on the a priori assumptions.

\subsection{Main contributions of this paper}

In this work, we present a general framework to solve infinite-dimensional inverse problems when only a finite number of measurements is available.\footnote{After the first version of this preprint was published, several works on this topic have appeared \cite{harrach-2019,alberti2020calderon,alberti2020inverse}.}

In Section~\ref{sec:lip} we prove a Lipschitz stability estimate for inverse problems with finite  measurements. We also give an explicit criterion to choose the type and number of measurements depending on the a priori assumptions on the unknown, namely on the space $W$. The result is obtained with functional analytic techniques and it takes inspiration from the general Lipschitz stability results of \cite{bacchelli2006, bourgeois2013} as well as from the already mentioned works on EIT  \cite{alberti2018,Harrach_2019}.

In Section~\ref{sec:rec} we derive a globally convergent reconstruction algorithm for the inverse problem by combining the Lipschitz stability obtained in the previous section with the results of \cite{dehoop2012} (summarized in Appendix~\ref{sec:app}), which guarantee local convergence for the associated Landweber iteration. We show how a good initial guess can be recovered by constructing a lattice in the (finite-dimensional) unknown space.

In Section~\ref{sec:exa}, we apply the abstract results to three inverse problems, EIT, inverse scattering and quantitative photoacoustic tomography. For EIT, we show a Lipschitz stability estimate from a finite number of measurements independent of the unknown, under the assumption that a Lipschitz stability estimate holds in the case of full measurements. Thus, this covers a large class of finite-dimensional spaces of conductivities. Moreover, an explicit criterion is given in order to choose the number of measurements. In particular, we consider the special case of conductivities on the unit disk: assuming that a Lipschitz stability estimate for the full Neumann-to-Dirichlet map holds with a constant $C$, we show that Lipschitz stability also holds with only a finite number of trigonometric current patterns proportional to $C^2$.

Finally, in Section~\ref{sec:conclusions} we provide some concluding remarks and discuss future research directions.

\subsection{Connections with machine learning}
In recent years, {\it machine learning} has arisen as a data-driven alternative to classical inversion methods, and has shown tremendous improvements in dealing with the ill-posedness of several inverse problems \cite{arridge2019,Lucas18,McCann17,2020-willet-etal}.
The results presented in this paper may be used in connection with machine learning techniques in several ways.
\begin{itemize}
\item Our main stability result holds under the assumption that the unknown belongs (or is close) to a known finite-dimensional subspace $W$ of a Banach space. In case this subspace is not known, it could be estimated by using, for instance, dictionary learning techniques \cite{mairal2011task}.
\item The main stability estimate of this work shows the Lipschitz continuity of an inverse mapping. This property allows the inverse to be approximated with neural networks \cite{yarotsky2017error}, even when $W$ is not known. The explicit discretisation that we use in the stability result would help in designing the architecture of such a network.
\item The iterative reconstruction algorithm presented in Section~\ref{sec:rec} can be \textit{unrolled} and seen as a neural network as in \cite{adler2017solving,adler2018learned}. Also in this context the subspace $W$ can be learned, namely the projection operator $P_W$ can be replaced by weights of a deep network.
\end{itemize}

\section{Lipschitz stability estimate}\label{sec:lip}

\subsection{Abstract inverse problem}
Let $X$ and $Y$ be Banach spaces with norms $\| \cdot \|_X$ and $\| \cdot \|_Y$. We denote the space of continuous linear operators from  $X$ to $Y$ by $\L(X,Y)$.
Let $A \subseteq X$ be an open set and $F \colon A \to Y$ be a map of class $C^1$. In other words,  $F$ is Fréchet differentiable at every $x \in A$, namely  there exists $F'(x)  \in \L(X,Y)$ such that
\[
\lim_{h \to 0}\frac{\|F(x+h)-F(x)-F'(x) h\|_Y}{\|h\|_X}=0,
\]
and  $F'$ is continuous. We shall refer to $F$ as the forward map.

The focus of this work is the following inverse problem: given $F(x)$, recover $x$. Very often, even if the map $F$ is injective, its inverse $F^{-1}$ is not continuous, and so the inverse problem is ill-posed.

\subsection{Lipschitz stability with infinite-dimensional measurements}
We are interested in a Lipschitz stability estimate for this general inverse problem. One way to obtain it, without imposing strong assumptions on the map $F$, is to assume that the unknown $x$ belongs to a finite-dimensional subset $W$ of $X$.  We recall here a general result obtained in \cite[Proposition 5]{bacchelli2006} and \cite[Theorem 1]{bourgeois2013} in a similar setting.

\begin{theorem}\label{theo:lip}
Let $X$ and  $Y$ be Banach spaces. Let $A \subseteq X$ be an open subset, $W \subseteq X$ be a finite-dimensional subspace and $K \subseteq W \cap A$ be a compact and convex subset. 
Let $F \in C^1(A,Y)$ be such that $F|_{W \cap A}$ and $F'(x)|_{W}$, $x \in W \cap A$, are injective.

Then there exists a constant $C> 0$ such that
\begin{equation}\label{est:lipinf}
\|x_1-x_2\|_X \leq C \| F(x_1)-F(x_2)\|_Y,\qquad x_1,x_2 \in K.
\end{equation}
\end{theorem}

The constant $C$ can be explicitly estimated in terms of the lower bound of the Fréchet derivative and the moduli of continuity of $(F|_{K})^{-1}$ and $F'$, as done in \cite{bacchelli2006}, or with ad-hoc techniques depending on the specific problem as in \cite{rondi2006}. In general, for ill-posed problems, $C\to +\infty$ as $\dim W\to+\infty$ (which is consistent with the results of \cite{1980-seidman}), and so this approach works well for low-dimensional spaces $W$ only.

Estimate \eqref{est:lipinf} shows that even a severely ill-posed problem becomes stable under the above finite dimensionality assumptions. Nevertheless, this is still unsatisfactory because we are recovering a finite-dimensional quantity $x$ from infinite-dimensional (and so, infinitely many scalar) measurements $F(x)$.

\subsection{Lipschitz stability with finitely many measurements}

Instead of measuring the infinite-dimensional data $F(x)$, we now suppose to have the measurements
\[
Q_N F(x)
\]
for $N\in\N$ large enough (depending explicitly on $K$ and $C$, as we will see), where $Q_N\colon Y\to Y$ are (finite-rank) bounded operators approximating the identity  $I_Y$ as $N\to +\infty$. The precise characterization of the convergence $Q_N\to I_Y$ is as follows.
\begin{hypothesis}\label{hyp}
Let  $W$ be a subspace of $X$ and $K \subseteq W \cap A$. For $N\in\N$, let $Q_N\colon Y\to Y$ be bounded linear maps. We assume that there exists a subspace
 $\tilde Y$ of $Y$ such that
 \begin{enumerate}[(a)]
 \item\label{ass:a} for every $\xi\in K$
\[
\ran \bigl(F'(\xi)|_W\bigr)\subseteq \tilde Y
\]
(if $F$ is linear, this is equivalent to requiring $\ran F|_W\subseteq\tilde Y$);
\item\label{ass:b} $\norm{Q_N}_{Y\to Y}\le D$ for every $N\in\N$, for some $D\ge 1$;
\item\label{ass:c} and, as $N\to +\infty$,
\[
Q_N|_{\tilde Y}\to I_{\tilde Y}
\]
with respect to the strong operator topology.
\end{enumerate}
\end{hypothesis}

\begin{remark*}
Several works study inverse problems with finitely-many measurements in the Bayesian approach to discretisation errors \cite{kaipio-somersalo-2007,lassas-etal-2009,stuart-etal-2012,kekkonen-etal-2014}, but they only consider asymptotic estimates (as $N$ goes to infinity) and are mostly for linear problems.

The framework considered here is similar to that of multilevel regularization methods, where discretization in the image space is used as a regularizer \cite{scherzer_book}. More precisely, it is possible to interpret this hypothesis as a smoothing property (cfr.\ the regularization theory by discretization \cite[Sections~4.3 and 5.1]{scherzer_book}).
\end{remark*}

Let us now list some important examples of operators $Q_N$. These are applied to realistic scenarios in Section~\ref{sec:exa}. \smallskip

Any family of bounded operators $Q_N$ such that $Q_N\to I_Y$ strongly may be considered (in particular, we do not require convergence with respect to the operator norm). In the next example, we discuss how this situation can arise in practice.

\begin{example}[Projections onto finite-dimensional subspaces]\label{ex:hil}

Let the data space $Y$ be an infinite-dimensional separable Hilbert space and $\{e_j\}_{j\in\N}$ be an orthonormal basis (ONB) of $Y$. Then $Q_N \colon Y \to Y$ can be taken as the projection onto the space generated by the first $N$ elements of the basis, i.e.\ $Q_N y = \sum_{j=1}^N \langle y,e_j\rangle e_j$.

More generally, assume that $Y$ has the following multi-resolution structure. Let $\{G_N\}_{N\in\N}$ be a family of subspaces of $Y$ satisfying:
\begin{enumerate}
\item each $G_N \subseteq Y$ is a finite-dimensional subspace;
\item $G_N \subseteq G_{N+1}$ for every $N \in \N$;
\item and the spaces are exhaustive, i.e. $\overline{\bigcup_{N\in \N}G_N} = Y$.
\end{enumerate}
Let $P_N\colon Y\to Y$ be the orthogonal projection onto $G_N$. We have that $P^*_N = P_N$ and if $G_N = \spann\{g_1,\dots,g_{d_N}\}$, where $d_N = \dim G_N$, the measurements
\[
\langle F(x),g_k\rangle, \quad k = 1,\ldots, d_N,
\]
determine the Galerkin projection $P_N F(x)$. Thus this represents a model for a finite number of measurements from the full data $F(x)$. In this case we can choose $Q_N = P_N$.

For example, the projections $Q_N$ may represent a low-pass filter in frequency (modelling sensors up to a certain bandwidth) or in wavelet scale.
\end{example}

Next, we consider a particular case of multi-resolution structure in reproducing kernel Hilbert spaces, in order to view the projected measurements $Q_NF(x)$ as a finite pointwise	 sampling of the full measurements $F(x)$.
\begin{example}[Finite sampling in reproducing kernel Hilbert spaces]\label{ex:rkhs}
Let $Y$ be a reproducing kernel Hilbert space (RKHS) of continuous functions on a separable topological space $\mathcal{A}$ \cite{1950-Aronszajn} (for instance, $Y$ can be a sufficiently smooth Sobolev space, possibly on a manifold \cite{devito-2020}, which is a common setup in inverse problems). Roughly speaking, this means that the evaluation functionals $f\mapsto f(a)$ are continuous for every $a\in \mathcal{A}$, so that by the Riesz representation theorem we can write the reproducing property
\[
f(a)=\langle f,k_a\rangle_Y,\qquad f\in Y,\,a\in\mathcal{A},
\]
for a suitable $k_a\in Y$. The kernel of the RKHS is then
\[
k\colon\mathcal{A}\times\mathcal{A}\to\C,\qquad k(a,b)=\langle k_b,k_a\rangle_Y.
\]

Let $\{a_j:j\in\N\}$ be a dense set of $\mathcal{A}$ (assume that $a_j\neq a_k$ if $j\neq k$): its existence follows from the separability of $\mathcal{A}$. Set
\[
G_N =\operatorname{span}\{k_{a_j}:j=1,\dots,N\},\qquad N\in\N.
\]
Let us now verify that the conditions of Example~\ref{ex:hil} are satisfied. By construction, $G_N$ is a finite-dimensional subspace of $Y$ and $G_N\subseteq G_{N+1}$ for every $N\in\N$. It remains to show that $\overline{\bigcup_{N\in \N}G_N} = Y$. Take $f\in Y$ such that $\langle f,g\rangle_Y = 0$ for every $g\in \bigcup_{N\in \N}G_N$. Then
\[
f(a_j)=\langle f,k_{a_j}\rangle =0,\qquad j\in\N,
\]
and so $f=0$ because it is continuous and $\{a_j:j\in\N\}$ is  dense in  $\mathcal{A}$. The claim is proved.

As in Example~\ref{ex:hil}, we can let $Q_N\colon Y\to Y$ be the projection onto $G_N$. Since
\[
\langle Q_N f,k_{a_j}\rangle=\langle f,k_{a_j}\rangle =f(a_j) ,\qquad j=1,\dots,N,
\]
the pointwise values $f(a_j)$ for $j=1,\dots,N$ allows us to recover $Q_N f$. Furthermore, since $G_N$ is finite-dimensional, there exists $C_N>0$ such that 
\begin{equation}\label{eq:stable_sampling}
\|Q_N f\|_Y \le C_N \|\bigl(f(a_1),\dots,f(a_N)\bigr)\|_2,\qquad f\in Y.
\end{equation}
In other words, the reconstruction of $Q_N f$ from the sampled values of $f$ is stable.
\end{example}

Let us now consider a third example that is adapted to inverse boundary value problems, in which the measurements themselves are operators. In this case, the maps $Q_N$ do not converge strongly to the identity on the whole $Y$.

\begin{example}[Projections with operator-valued measurements]\label{ex:gal}
Let $Y = \L (Y^1,Y^2)$ be the space of bounded linear operators from $Y^1$ to $Y^2$, where $Y^1$ and $Y^2$ are Banach spaces. Let  $P^k_N \colon Y^k \to Y^k$ be bounded maps such that $P^{2}_N \to I_{Y^k}$ and $(P^{1}_N)^* \to I_{Y^k}$ strongly as $N\to +\infty$ (for instance, when $Y^k$ is a Hilbert space, $P^k_N$ may be chosen  as the projection onto the $N$-th subspace of an exhaustive chain, as in Example~\ref{ex:hil}). 

Let us now take $\tilde Y=\{T\in Y:\text{$T$ is compact}\}$ and assume that
\begin{equation}\label{eq:compact}
F'(\xi)\tau \colon Y^1\to Y^2\;\; \text{is compact for every $\xi \in K, \tau \in W$,}
\end{equation}
which is satisfied in many cases of interest. It is worth observing that this condition is implied by
\begin{equation*}\label{eq:compact2}
F(x_1)-F(x_2)\;\; \text{is compact for every $x_1,x_2 \in A$,}
\end{equation*}
since 
$
F'(\xi)\tau=\lim_{h \to 0} \frac{F(\xi +h\tau) -F(\xi)}{h}
$
is the operator-norm limit of compact operators, and so it is compact.

Let  $Q_N \colon Y \to Y$  be the maps defined by
\[
Q_N (y) = P^2_N y P^1_N,\qquad y \in Y.
\]

Even if $Q_N \not\to I_Y$ strongly, let us show that Hypothesis~\ref{hyp} is satisfied. Condition \ref{ass:a} follows immediately from \eqref{eq:compact}. By the uniform boundedness principle the operators $P^k_N$ are uniformly bounded, and so  \ref{ass:b} is satisfied with $D=\sup_N \|P^1_N\|_{Y^1\to Y^1}\|P^2_N\|_{Y^2\to Y^2}$. It remains to verify \ref{ass:c}. For 
 $T \in\tilde Y$ we have
\begin{equation}\label{eq:RN}
\begin{split}
\| T-Q_N T \|_Y &= \| T -  P^2_N T+ P^2_N T- P^2_N T P^1_N \|_Y \\
& \leq \|(I_{Y^2}-P_N^2)T\|_Y + \|P^2_NT(I_{Y^1}-P_N^1)\|_Y\\ 
& \leq \|(I_{Y^2}-P_N^2)T\|_Y + \|P^2_N\|_{Y^2\to Y^2}\|T(I_{Y^1}-P_N^1)\|_Y.
\end{split}
\end{equation}
Since $P_N^{2} \to I_{Y^2}$ and $(P_N^{1})^* \to I_{Y^1}$ strongly, and $T \colon Y^1 \to Y^2$ is compact, by a standard result in functional analysis (see Lemma~\ref{lem:comp} below) we have that the right hand side converges to $0$ as $N\to +\infty$, and \ref{ass:c} follows.
\end{example}

Our main result states that it is indeed possible to obtain Lipschitz stability from a finite number of measurements, obtained by composing $Q_N$ with the full measurements, at the price of a slightly larger Lipschitz constant. In particular, we also obtain a uniqueness result with finite measurements.

\begin{theorem}\label{theo:main}
Let $X$ and $Y$ be Banach spaces. Let $A \subseteq X$ be an open subset, $W\subseteq X$ be a subspace and $K \subseteq W\cap A$ be a compact and convex subset. Let $F \in C^1(A,Y)$ be a Lipschitz map satisfying the Lipschitz stability estimate
\[
\|x_1-x_2\|_X \leq C \| F(x_1)-F(x_2)\|_Y, \qquad x_1,x_2 \in K,
\]
for some $C\ge 1$.
Let $Q_N\colon Y\to Y$ be bounded linear maps for $N\in\N$ satisfying Hypothesis~\ref{hyp}.
\begin{enumerate}[(i)]
\item  \label{cond:fin} If $W$ is finite-dimensional, then 
\[
\lim_{N \to +\infty}s_N = 0,\qquad s_N = \sup_{\xi \in K}\|(I_Y-Q_N) F'(\xi)\|_{W \to Y}.
\]
\item\label{cond:sn} If $s_N \leq \frac{1}{2C}$, then for every $x_1,x_2\in A$
\begin{equation}\label{eq:lip2C}
\|x_1-x_2\|_X \leq 2 C \| Q_N(F(x_1))-Q_N(F(x_2))\|_Y + 3\, CDL (d(x_1,K)+d(x_2,K)), 
\end{equation}
where $L\ge 1$ is the Lipschitz constant of $F$.
\end{enumerate}
\end{theorem}

\begin{remark}\label{rem:1}
Even though  \ref{cond:fin} and \ref{cond:sn} combined clearly imply that there exists $N\in\N$ for which \eqref{eq:lip2C} holds true, we decided to separate the two conditions since in many cases it is possible to show directly $s_N \leq \frac{1}{2C}$, and this allows to derive explicit bounds on $N$ (as in Section \ref{sub:eit}). More precisely, the inequality
\begin{equation}\label{cond:sn2}
\sup_{\xi \in K}\|(I_Y-Q_N) F'(\xi)\|_{W \to Y} \leq \frac{1}{2C}
\end{equation}
gives  an explicit relation between the number of measurements/discretization parameter $N$, the space $W$ and the Lipschitz stability constant $C$ of the full-mea\-su\-re\-ment case. It is worth observing that this condition may be seen as a nonlinear version of the \textit{stable sampling rate} and the \textit{balancing property}, which were introduced in the context of generalized sampling and compressed sensing for infinite-dimensional linear problems \cite{AH-2012,AHP-2013,AH,2017-adcock-hansen-poon-roman}.
\end{remark}

\begin{remark}\label{rem:2}
When the mismodeling errors $d(x_i,K)$ are zero, estimate \eqref{eq:lip2C} is exactly a Lipschitz stability estimate. Otherwise, it is natural to have an error term depending on the distance to $K$. We want to underline that, in our framework, this error term is amplified by the Lipschitz constant $C$, which may be very large for severely ill-posed problems. This is due to the linearization techniques in our proof and it might be a price to pay to obtain a general result such as Theorem~\ref{theo:main}. However, in \cite{alberti2018,alberti2020calderon} we were able to obtain sharper estimates with respect to the mismodeling error for Calder\'on's inverse conductivity problem, thanks to specific nonlinear inversion methods.
\end{remark}

Combining Theorems \ref{theo:lip} and \ref{theo:main} we obtain the following corollary for the cases discussed in Examples~\ref{ex:hil} and \ref{ex:gal}.
\begin{corollary} \label{cor:main}
Let $X$ and $Y$ be Banach spaces. Let $A \subseteq X$ be an open subset, $W$ be a finite-dimensional subspace of $X$ and $K$ be a compact and convex subset of $W \cap A$. 
Let $F \in C^1(A,Y)$ be such that $F|_{W \cap A}$ and $F'(x)|_{W}$, $x \in W \cap A$, are injective. Let $Q_N\colon Y\to Y$ be bounded linear maps for $N\in\N$. Assume one of the following:
\begin{enumerate}[(a)]
\item \label{cond:yhil} $Q_N\to I_Y$ strongly as $N\to +\infty$ (as, e.g., in Example \ref{ex:hil});
\item \label{cond:ylin} $Y = \L (Y^1,Y^2)$, with $Y^1$ and $Y^2$ Banach spaces,   $Q_N$ is given as in  Example~\ref{ex:gal}, and $F'(\xi)\tau \colon Y^1\to Y^2$ is compact for every $\xi \in A$ and $\tau \in W$.
\end{enumerate}
Then there exist $N \in \N$ and $C>0$ such that
\begin{equation}
\|x_1-x_2\|_X \leq C \| Q_N(F(x_1))-Q_N(F(x_2))\|_Y, \qquad x_1,x_2 \in K.
\end{equation}
\end{corollary}

This corollary shows that in the rather common cases discussed in Examples~\ref{ex:hil} and \ref{ex:gal}, under basically the same conditions of Theorem \ref{theo:lip} we obtain Lipschitz stability from a finite number of measurements.  

Let us now prove the main result.

\begin{proof}[Proof of Theorem \ref{theo:main}]
We prove the two parts of the statement separately.

\textit{Proof of \ref{cond:fin}.} Let $K' = K \times (W \cap S(0,1))$, where $S(0,1)$ is the unit sphere in $X$ centered in $0$, and
\[
g \colon K' \to Y, \quad g(\zeta) = F'(\xi)\tau,\quad \zeta = (\xi,\tau).
\]
We first show that $g$ is continuous. Given a converging sequence $\zeta_n= (\xi_n,\tau_n) \to \zeta = (\xi,\tau)$ in $K'$, we have
\begin{align*}
\|g(\zeta_n)-g(\zeta)\|_Y &= \|F'(\xi_n)\tau_n - F'(\xi)\tau\|_Y\\
&\leq \|(F'(\xi_n) - F'(\xi))\tau_n\|_Y + \|F'(\xi)(\tau_n-\tau)\|_Y\\
&\leq \|F'(\xi_n) - F'(\xi)\|_{X \to Y}\|\tau_n\|_X + \|F'(\xi)\|_{X\to Y} \|\tau_n-\tau\|_X.
\end{align*}
Using the fact that $F \in C^1(A,Y)$, $\|\tau_n\|_X = 1$ and $\|\tau_n-\tau\|_X \to 0$ we have that $\|g(\zeta_n)-g(\zeta)\|_Y \to 0$, and so $g$ is continuous. Set $R_N=I_Y-Q_N$. Thus the map
\[
K' \ni (\xi,\tau)   \mapsto \|R_N F'(\xi)\tau\| = \|R_N g(\zeta)\|\in [0,+\infty)
\]
is continuous on the compact set $K'$, so there exists $\zeta_N \in K'$ such that
\[
s_N = \sup_{\xi \in K}\|R_N F'(\xi)\|_{W \to Y} =\sup_{\zeta \in K'} \|R_N g(\zeta)\|_Y = \|R_N g(\zeta_N)\|_Y.
\]
Let $(s_{N_j})_j$ be a subsequence of $(s_N)_N$: by a classical result in general topology, it is enough to show that $(s_{N_j})_j$ has a subsequence converging to $0$. The set $K'$ is compact, and so $\zeta_{N_j} \to \zeta_*$ in $K'$ for  a subsequence (with an abuse of notation, we do not specify the second subsequence), so
\begin{align*}
\|R_{N_j} g(\zeta_{N_j})\|_Y &\leq \|R_{N_j} (g(\zeta_{N_j})-g(\zeta_*))\|_Y + \|R_{N_j} g(\zeta_*)\|_Y\\
&\leq \|R_{N_j} \|_{\tilde Y \to Y} \|g(\zeta_{N_j})-g(\zeta_*)\|_Y+\|R_{N_j} g(\zeta_*)\|_Y,
\end{align*}
where we have used Hypothesis~\ref{hyp}\ref{ass:a}.
Since $\|Q_{N_j}\|_{\tilde Y \to Y} \leq D$ by Hypothesis~\ref{hyp}\ref{ass:b}, we have $\|R_{N_j} \|_{\tilde Y \to Y}  \leq D+1$. Moreover, $\|g(\zeta_{N_j})-g(\zeta_*)\|_Y \to 0$ as $j \to +\infty$ since $g$ is continuous. Finally, $\|R_{N_j} g(\zeta_*)\|_Y \to 0$ as $j \to +\infty$ by Hypothesis~\ref{hyp}\ref{ass:c}. This shows that $s_{N_j}\to 0$ as $j \to +\infty$, as desired.

\textit{Proof of \ref{cond:sn}.} Let $N$ be such that $s_N\le \frac{1}{2C}$.
Take $x_1,x_2\in A$ and let $x_1^K,x_2^K\in K$ be such that
\begin{equation}\label{eq:distance}
d(x_j,K)=\|x_j-x_j^K\|_X,\qquad j=1,2.
\end{equation}
(The existence of these points is guaranteed by the compactness of $K$.) By the Lipschitz stability estimate of $F$ we have
\begin{equation}\label{pf:ineq}
\begin{split}
\|x_1^K-x_2^K\|_X  &\leq C \|F(x_1^K)-F(x_2^K)\|_Y \\
&\leq C\| Q_N(F(x_1^K)-F(x_2^K))\|_Y+C \|R_N(F(x_1^K)-F(x_2^K))\|_Y.
 \end{split}
\end{equation}

Consider now $f \colon A \to Y$, $f(x)  = R_N F(x)$. Then it is easy to verify that its Fréchet derivative satisfies $f'(x) = R_N F'(x)$, because $R_N$ is linear. Then, by the mean value theorem for Gateaux differentiable functions between Banach spaces (see \cite[Theorem 1.8]{ambrosetti1995} for instance) and since $K$ is convex, there exists $\xi \in K$ such that
\[
\frac{\|R_N(F(x_1^K)-F(x_2^K))\|_Y}{\|x_1^K-x_2^K\|_X}\leq \frac{\|R_N F'(\xi)(x_1^K-x_2^K) \|_Y}{\|x_1^K-x_2^K\|_X}\leq s_N \leq \frac{1}{2C}.
\]
Plugging this estimate into inequality \eqref{pf:ineq} yields
\[
\begin{split}
\|x_1^K-x_2^K\|_X &\leq 2 C\| Q_N(F(x_1^K)-F(x_2^K))\|_Y\\
 &\leq 2C\|Q_N(F(x_1)-F(x_2))\|_Y + 2C\|Q_N(F(x_1)-F(x_1^K))\|_Y \\
 &\quad + 2C\|Q_N(F(x_2)-F(x_2^K))\|_Y \\
  &\leq 2C\|Q_N(F(x_1)-F(x_2))\|_Y +2CDL (\|x_1-x_1^K\|_X+\|x_2-x_2^K\|_X) ,
\end{split}
\]
because $\|Q_N\|_{Y\to Y} \leq D$ and $L$ is the Lipschitz constant of $F$. As a consequence, thanks to \eqref{eq:distance} we readily obtain
\[
\begin{split}
\|x_1-x_2\|_X  &\leq\|x_1^K-x_2^K\|_X + \|x_1-x_1^K\|_X+\|x_2-x_2^K\|_X\\
&\leq 2C\|Q_N(F(x_1)-F(x_2))\|_Y +(2CDL+1) (d(x_1,K)+d(x_2,K)).
\end{split}
\]
This immediately gives the desired Lipschitz stability estimate \eqref{eq:lip2C}.
\end{proof}

We finish this section with a technical lemma used in Example~\ref{ex:gal}.

\begin{lemma}\label{lem:comp}
Let $Y^1$ and $Y^2$ be two Banach spaces, $T \colon Y^1 \to Y^2$ be a compact operator and $S^k_N \colon Y^k \to Y^k$, $k=1,2$, be such that $S^{2}_N \to 0$ and $(S^1_N)^* \to 0$ strongly as $N \to +\infty$. Then
\begin{enumerate}[(i)]
\item \label{conv:1} $\|S^2_N T\|_{\L(Y^1,Y^2)} \to 0$ as $N \to +\infty$,
\item \label{conv:2} $\|T S^1_N\|_{\L(Y^1,Y^2)}\to 0$ as $N \to +\infty$.
\end{enumerate}
\end{lemma}

\begin{proof}
We start from \ref{conv:1}. Assume by contradiction that $\|S^2_N T\|_{\L(Y^1,Y^2)} \not\to 0$. Then there exists $\varepsilon > 0$ such that, for a subsequence, we have
\[
\sup_{\|x\|_{Y^1} = 1} \|S^2_N T x \|_{Y^2} > \varepsilon, \qquad  N \in \N.
\]
Thus, for each $N\in \N$ there exists $x_N\in Y^1$ such that $\|x_N\|_{Y^1} = 1$ and $\|S^2_N T x_N\|_{Y^2} > \varepsilon$. Since $\{x_N\}_{N \in \N}$ is bounded and $T$ is compact, there exists $y \in Y^2$  such that, for a subsequence, $T x_{N} \to y$ as $N \to +\infty$. Since $S^2_N y' \to 0$ for every $y' \in Y^2$ we have
\[
\sup_{N \in \N}\|S^2_N y'\|_{Y^2}<+\infty, \qquad y' \in Y^2,
\]
and by the uniform boundedness principle we obtain
\[
\sup_{N \in \N} \|S^2_N\|_{\L(Y^2,Y^2)} \leq C,
\]
for some $C > 0$. Thus
\begin{align*}
\varepsilon &< \|S^2_{N} T x_{N}\|_{Y^2}\\
&\leq \| S^2_{N} y\|_{Y^2} + \|S^2_{N} (Tx_{N} - y)\|_{Y^2}\\
&\leq  \| S^2_{N} y\|_{Y^2} + C\|Tx_{N} - y\|_{Y^2}
\end{align*}
which is a contradiction, since the last term goes to zero as $N \to +\infty$.

The proof of \ref{conv:2} follows from \ref{conv:1} using the fact that 
\[
\|TS^1_N\|_{\L(Y^1,Y^2)} = \| (S^1_N)^* T^*\|_{\L(Y^2,Y^1)},
\]
the compactness of $T^*$ and that $(S^1_N)^* \to 0$ strongly.
\end{proof}

\section{Reconstruction algorithm}\label{sec:rec}
The Lipschitz stability estimates presented in Theorem~\ref{theo:main} and Corollary~\ref{cor:main} can be used to design a  reconstruction algorithm.
We slightly strengthen  the assumptions of  Theorem \ref{theo:main} (regarding the regularity of $F$ and the map $Q$), and let
\begin{itemize}
\item $X$ and $Y$ be Banach spaces;
\item $A\subseteq X$ be an open set;
\item $W\subseteq X$ be a finite-dimensional subspace;
\item $F\in C^{1}(A,Y)$ be such that $F|_{W\cap A}$ and $(F|_{W\cap A})'$ are Lipschitz continuous;
\item $Q\colon Y\to Y$ be a continuous finite-rank operator;
\item $K\subseteq W\cap A$ be a compact set;
\item and $C>0$ be a positive constant such that
\begin{equation}\label{eq:rec_stab}
\|x_1-x_2\|_X \leq 2 C \| Q(F(x_1))-Q(F(x_2))\|_Y, \qquad x_1,x_2 \in W\cap A.
\end{equation}
\end{itemize}

Let $x^\dag\in K$ be the unknown signal and $y=Q(F(x^\dag))\in Y$ denote the corresponding measurements. We now derive a global reconstruction algorithm which allows for the recovery of $x^\dag$ from the knowledge of $y$. We assume here, for simplicity, that we have noiseless data at our disposal. However, we expect the extension to the noisy case to be possible by using the results of \cite{2021-mittal-giri}.
\subsection{Local reconstruction}\label{sub:local}

We first discuss how to reconstruct $x^\dag$ by means of an iterative method \cite{dehoop2012,dehoop2015} (see Appendix~\ref{sec:app}), provided that a good approximation $x_0$ of $x^\dag$ is known. 

The domain $W$ and the range of $Q$ are finite dimensional, and so they are isomorphic to finite-dimensional euclidean spaces. In particular, without loss of generality, we can assume that $X$ and $Y$ are Hilbert spaces.

We wish to apply Proposition~\ref{prop:land} to $Q\circ F|_{W\cap A}$ and $K$: let    $\rho,\mu,c\in (0,1)$ be given as in the statement of Proposition~\ref{prop:land}. Let $x_0\in B_X(x^\dag,\rho)\cap K$ be the initial guess  of the Landweber iteration with stepsize $\mu$ 
\begin{equation*}
x_{k+1}=x_{k}-\mu (F|_{W\cap A})'(x_{k})^{*}Q^*\left(Q(F(x_{k}))-y\right),\qquad k\in\N,
\end{equation*}
related to the minimization of
\begin{equation*}
\min_{x\in W\cap A} \norm{Q(F(x))-y}_Y^2.
\end{equation*}
By Proposition~\ref{prop:land}, $x_k\to x^\dag$. More precisely, the convergence rate is given by
\[
\norm{x_k-x^\dag}_X\le \rho c^k,\qquad k\in\N.
\]

In view of the multi-level scheme of \cite{dehoop2015}, with Lipschitz stability constants $C_\alpha$ depending on each level $\alpha \in \N$, it is interesting to note that in this case the projection $Q_N$ would also depend on $\alpha$ through  condition \eqref{cond:sn2}. Thus, the number $N = N_\alpha$ of measurements would grow at each iteration depending on the Lipschitz constant $C_\alpha$.

\subsection{Getting the initial guess $x_0$}
Let us now discuss how to find the initial guess $x_0$ for the Landweber iteration, namely an approximation of $x^\dag$ so that $\|x_0-x^\dag\|_X<\rho$.

The forward map $F$ is Lipschitz continuous on $K$ by assumption, and so there exists $L>0$ such that
\begin{equation}\label{eq:forward}
\norm{F(x_1)-F(x_2)}_Y\le L\norm{x_1-x_2}_X,\qquad x_1,x_2\in K.
\end{equation}
Since $K$ is compact, we can find a finite lattice $\{x^{(i)}:i\in I\}\subseteq K$ such that
\begin{equation}\label{eq:lattice}
K\subseteq\bigcup_{i\in I} B_X\left(x^{(i)},\frac{\rho}{2L C\|Q\|_{Y\to Y}}\right),
\end{equation}
where $\rho>0$ is given by Proposition~\ref{prop:land} and $C$ is given by \eqref{eq:rec_stab}.
\begin{lemma}
Under the above assumptions, we have:
\begin{enumerate}[(a)]
\item\label{lattice1}  there exists $i\in I$ such that
\begin{equation}\label{eq:i}
\bigl\lVert Q(F(x^\dag))-Q(F(x^{(i)}))\bigr\rVert_Y < \frac{\rho}{2C};
\end{equation}
\item\label{lattice2} if \eqref{eq:i} holds true for some $i\in I$, then
\[
\|x^{(i)}-x^\dag\|_X<\rho.
\]
\end{enumerate}
\end{lemma}
\begin{proof}
\textit{Proof of \ref{lattice1}.} By \eqref{eq:lattice} there exists $i\in I$ such that $\norm{x^\dag-x^{(i)}}_X<\frac{\rho}{2L C\|Q\|_{Y\to Y}}$. Thus, \eqref{eq:i} is an immediate consequence of \eqref{eq:forward}.

\textit{Proof of \ref{lattice2}.} By \eqref{eq:rec_stab} we have
\begin{equation*}
\|x^{(i)}-x^\dag\|_X \leq 2 C \| Q(F(x^{(i)}))-Q(F(x^\dag))\|_Y<\rho.\qedhere
\end{equation*}
\end{proof}

Thanks to this lemma, it is enough to let the initial guess $x_0$ be chosen as one of the elements of the lattice $\{x^{(i)}\}_{i\in I}$ for which \eqref{eq:i} is satisfied. Indeed, one such element exists by part \ref{lattice1} and, by part \ref{lattice2}, we have the desired approximation
\[
\|x_0-x^\dag\|_X<\rho.
\]
It is worth observing that this method requires  the ``physical'' measurements $Q(F(x^\dag))$ and to compute $Q(F(x^{(i)}))$ for $i\in I$ ``offline'', which can be done in parallel.

\subsection{Global reconstruction}
These two steps  may be combined to obtain a global reconstruction algorithm; see Algorithm~\ref{alg:rec}.

\begin{algorithm}
  \caption{Reconstruction of $x^\dag$ from $Q(F(x^\dag))$}
  \label{alg:rec}
  \begin{algorithmic}[1]
    \STATE Input $X$, $Y$, $W$, $K$, $Q$, $F$, $Q(F(x^\dag))$, $L$, $\rho$, $\mu$, $C$ and $M$.
    \STATE Equip $W$ and $Q(Y)$ with equivalent euclidean scalar products.
    \STATE Find a finite lattice $\{x^{(i)}:i\in I\}\subseteq K$ so that \eqref{eq:lattice} is satisfied.
     \FOR {$i\in I$}
    \STATE Compute $Q(F(x^{(i)}))$.
    \IF{\eqref{eq:i} is satisfied} 
    \STATE Set $x_0=x^{(i)}$.
    \STATE Exit \textbf{for}.
    \ENDIF 
    \ENDFOR
    \FOR {$k=0,\ldots,M$}
     \STATE Set $x_{k+1}=x_{k}-\mu (F|_{W\cap A})'(x_{k})^{*}Q^*\left(Q(F(x_{k}))-y\right)$.
     \STATE Check the stopping criterion.
    \ENDFOR
    \STATE Output $x_{k+1}$.
  \end{algorithmic}
\end{algorithm}

\section{Examples}\label{sec:exa}

In this section we show how Theorem \ref{theo:main} and Corollary \ref{cor:main} can be used to derive Lipschitz stability estimates for several inverse problems with a finite number of measurements.

\subsection{Electrical impedance tomography}\label{sub:eit}

Consider a bounded Lipschitz domain $\Omega \subseteq \R^d$, with $d\ge 2$, equipped with an electrical conductivity $\sigma \in L_+^{\infty}(\Omega)$, where
\[
L_+^{\infty}(\Omega):=\{f\in L^\infty(\Omega):f \ge \lambda \text{ a.e.\ in $\Omega$, for some $\lambda>0$}\}.
\]
The corresponding Neumann-to-Dirichlet (ND) or current-to-voltage map is the operator $\ND_{\sigma}\colon L_\diamond^{2}(\partial \Omega) \to L_\diamond^{2}(\partial \Omega)$, defined by
\begin{equation}
\ND_{\sigma}(g) = u^g_\sigma|_{\partial \Omega},
\end{equation}
where $L^{2}_\diamond(\partial \Omega) = \{ f \in L^2(\partial \Omega) : \int_{\partial \Omega} f\, ds = 0\}$ and $u^g_\sigma$ is the unique $H^1(\Omega)$-weak solution of the Neumann problem for the conductivity equation
\begin{equation}\label{schr}
\left\{
\begin{array}{ll}
-\nabla \cdot (\sigma \nabla u^g_\sigma) = 0 \quad &\text{in } \Omega,\\
\sigma \displaystyle \frac{\partial u^g_\sigma}{\partial \nu} =g \quad &\text{on } \partial \Omega,
\end{array}
\right.
\end{equation}
where $\nu$ is the unit outward normal to $\partial \Omega$,
satisfying the normalization condition
$$\int_{\partial \Omega} u^g_\sigma\, ds = 0.
$$ 
The following inverse boundary value problem arises from this framework, see \cite{calderon1980,2002-borcea,2009-uhlmann} and references therein.

\vspace*{.2cm}

{\bf Inverse conductivity problem.} Given $\ND_{\sigma}$, find $\sigma$ in $\Omega$.

\vspace*{.2cm}

 It is well known that the knowledge of $\ND_\sigma$ determines $\sigma$ uniquely if $d=2$ \cite{Nachman1996,astala2006} or if $\sigma$ is smooth enough \cite{Sylvester1987,haberman2015,caro2016}. The inverse problem is severely ill-posed, and only logarithmic stability holds true \cite{alessandrini1988,mandache2001,clop2010,caro2013}.

In recent years several Lipschitz stability estimates have been obtained for this inverse problem under certain a priori-assumption on $\sigma$, such as for $\sigma$ piecewise constant \cite{2005-alessandrini-vessella}, piecewise linear \cite{alessandrini2017} or for $\sigma$ belonging to a finite-dimensional subspace of piecewise analytic functions \cite{Harrach_2019} (see \cite{gaburro2015,alessandrini2017u} for the anisotropic case). In all these cases, the conductivity is always assumed to lie in a certain finite-dimensional subspace $W$ of $L^\infty(\Omega)$, but the full boundary measurements are required, i.e.\ all possible combination of current/voltage data (see formula \eqref{eq:lip_eit_full} below).

We now show how to derive the same estimates with finitely many measurements by applying the results of Section~\ref{sec:lip}. Further, the reconstruction algorithm of Section~\ref{sec:rec} may be used to recover the unknown conductivity.

We now fix the main ingredients of the construction. Let
\begin{itemize}
\item $X=L^\infty(\Omega)$;
\item $Y=\mathcal{L}_c(L_\diamond^{2}(\partial \Omega),L_\diamond^{2}(\partial \Omega))$;
\item $A=L^\infty_+(\Omega)$;
\item $F(\sigma)=\ND_\sigma$, for $\sigma\in L^\infty_+(\Omega)$;
\item $W$ be a finite-dimensional subspace of $L^\infty(\Omega)$ such that
\begin{equation}\label{eq:lip_eit_full}
\norm{\sigma_1-\sigma_2}_\infty \le C \norm{\ND_{\sigma_1}-\ND_{\sigma_2}}_{L_\diamond^{2}(\partial \Omega)\to L_\diamond^{2}(\partial \Omega)},\qquad \sigma_1,\sigma_2\in K,
\end{equation}
for some $C>0$, where $K=W\cap L^\infty_\lambda(\Omega)$ and
\[
L^\infty_\lambda(\Omega)=\{f\in L^\infty_+(\Omega):\lambda^{-1}\le f\le\lambda\text{ a.e.\ in $\Omega$}\}
\]
for some $\lambda>0$ fixed a priori;
\item $P_N\colon  L_\diamond^{2}(\partial \Omega)\to L_\diamond^{2}(\partial \Omega)$ be bounded linear maps, $N\in \N$, such that $P_N=P_N^*$ and $P_N\to I_{L_\diamond^{2}(\partial \Omega)}$ strongly as $N\to+\infty$ (as, e.g., in Example~\ref{ex:hil});
\item and $Q_N y = P_NyP_N$ for $y\in Y$, as in Example~\ref{ex:gal}.

\end{itemize}

Under these assumptions, we have the following Lipschitz stability estimate with a finite number of measurements. This result extends \cite{Harrach_2019} to any subspace $W$ yielding \eqref{eq:lip_eit_full} and, in addition, it provides a constructive way to determine the parameter $N$.
\begin{theorem}\label{thm:eit}
Under the above assumptions, there exists $N \in \N$ such that
\begin{equation}\label{eq:lip_eit}
\norm{\sigma_1-\sigma_2}_\infty \le 2C\, \norm{P_N\ND_{\sigma_1}P_N-P_N\ND_{\sigma_2}P_N}_{L_\diamond^{2}(\partial \Omega)\to L_\diamond^{2}(\partial \Omega)},\qquad \sigma_1,\sigma_2\in K,
\end{equation}
where $C$ is given in \eqref{eq:lip_eit_full}.

The parameter $N$ can be chosen such that 
\begin{equation}\label{eq:eit_NC}
\norm{J(I_{L^2_\diamond(\partial\Omega)}-P_N)}_{L^2_\diamond(\partial\Omega)\to H^{-\frac12}_\diamond(\partial\Omega)}\le\frac{c(\Omega,\lambda)p}{C},
\end{equation}
where $c(\Omega,\lambda)$ is a constant depending only on $\Omega$ and $\lambda$, 
$J\colon L^2_\diamond(\partial\Omega)\to H^{-\frac12}_\diamond(\partial\Omega)$ is the canonical immersion, $H^{-1/2}_\diamond(\partial \Omega) = \{ f \in H^{-1/2}(\partial \Omega) : \int_{\partial \Omega} f\, ds = 0\}$, and $p=\sup_{N} \norm{P_N}_{L_\diamond^{2}(\partial \Omega)\to L_\diamond^{2}(\partial \Omega)}$. 

\end{theorem}

\begin{remark}\label{rem:finite}
It is worth analyzing the right hand side of \eqref{eq:lip_eit} in order to understand why it represents a finite number of measurements if the maps $P_N$ are properly chosen. Let $\{e_n\}_{n\in\N}$ be an ONB of $L_\diamond^{2}(\partial \Omega)$ and let $P_N$ be the orthogonal projection onto $\operatorname{span}\{e_1,\dots,e_N\}$ (similarly, we could also consider projections onto the $N$-th subspace of an exhaustive chain, as in Example~\ref{ex:hil}). In this case, we clearly have $p=1$, $P_N=P_N^*$ and that $P_N\to I_{L_\diamond^{2}(\partial \Omega)}$ strongly as $N\to+\infty$. 
Writing $g=\sum_{m=1}^{+\infty} g_m e_m$, we have
\begin{equation}\label{eq:spectral}
P_N \ND_\sigma P_N g = \sum_{n=1}^N \langle \ND_\sigma P_N g,e_n\rangle_{L_\diamond^{2}(\partial \Omega)} e_n
= \sum_{m,n=1}^N \langle \ND_\sigma e_m,e_n\rangle_{L_\diamond^{2}(\partial \Omega)} g_m e_n.
\end{equation}
Let $M_\sigma\in\R^{N\times N}$ be the  matrix defined by $(M_\sigma)_{m,n}=\langle \ND_\sigma e_m,e_n\rangle_{L_\diamond^{2}(\partial \Omega)}$: observe that this can be obtained by applying only the $N$ currents $e_1,\dots,e_N$ and measuring the corresponding voltages only up to frequency/scale $N$. By \eqref{eq:spectral}, estimate \eqref{eq:lip_eit} becomes
\begin{equation*}
\norm{\sigma_1-\sigma_2}_\infty \le 2C\, \norm{M_{\sigma_1}-M_{\sigma_2}}_{2},\qquad \sigma_1,\sigma_2\in K,
\end{equation*}
where $\norm{M}_2$ denotes the spectral norm of the matrix $M\in\R^{N\times N}$. In other words, we have a Lipschitz stability estimates with $N^2$ scalar measurements.
\end{remark}
Let us now comment on how large the parameter $N$ has to be chosen.

\begin{remark}
The parameter $N$ is given explicitly as a function of $C$ and $p$ in \eqref{eq:eit_NC}, up to a constant depending only on $\Omega$ and $\lambda$. Further, upper bounds on $C$ in \eqref{eq:lip_eit_full} depending on the dimension of $W$ are known in some particular cases, and so it is possible to give the number of measurements $N$ as a function of $\dim W$. For example, with piecewise constant conductivities, we have that $C$ depends exponentially on $\dim W$ \cite{rondi2006}.

Note that the norm $\norm{J(I_{L^2_\diamond(\partial\Omega)}-P_N)}_{L^2_\diamond(\partial\Omega)\to H^{-1/2}_\diamond(\partial\Omega)}$ depends exclusively on the domain $\Omega$, and can be explicitly estimated in some particular cases. In the following example, we consider the case when $\Omega$ is the unit ball in two dimensions and $P_N$ is the low-pass filter up to the frequency $N$.
\end{remark}
\begin{example}
Take $d=2$ and $\Omega=B(0,1)$. Parametrizing the boundary $\partial\Omega$ with the arc-length parameter $\theta$, for $s\in\{-\frac12,0\}$, using Fourier series we can write
\[
H^s_\diamond(\partial\Omega)=\Bigl\{\theta\in [0,2\pi)\mapsto g (\theta)=\sum_{n=1}^{+\infty} a_n \cos(n\theta)+b_n\sin(n\theta):\norm{g}_{H^s_\diamond(\partial\Omega)}<+\infty\Bigr\}
\]
where the norm is given by
$
\norm{g}_{H^s_\diamond(\partial\Omega)}^2=\sum_{n=1}^{+\infty} n^{2s}  (a_n^2+b_n^2). 
$
Define the projections $P_N\colon L^2_\diamond(\partial\Omega)\to L^2_\diamond(\partial\Omega)$ as 
\[
(P_N g)(\theta)=\sum_{n=1}^{N} a_n \cos(n\theta)+b_n\sin(n\theta).
\]
We clearly have that $P_N\to I_{L^2_\diamond(\partial\Omega)}$ strongly and that $\norm{P_N}_{L^2_\diamond(\partial\Omega)\to L^2_\diamond(\partial\Omega)}=1$ for every $N$, so that $p=1$. This is substantially in the same framework described in Remark~\ref{rem:finite}, the only difference being that the elements of the basis are considered in pairs. Thus, the measures here corresponds to $2N$ sinusoidal input currents and the related voltages measured up to frequency $N$. Thanks to the Nyquist-Shannon sampling theorem, the voltages may be measured only at a finite number of locations on $\partial\Omega$.

Observing that $(I_{L^2_\diamond(\partial\Omega)}-P_N)g=\sum_{n=N+1}^{+\infty} a_n \cos(n\cdot)+b_n\sin(n\cdot)$, we readily derive
\[
\norm{J(I_{L^2_\diamond(\partial\Omega)}-P_N)}^2_{L^2_\diamond(\partial\Omega)\to H^{-\frac12}_\diamond(\partial\Omega)} = \sup_{\norm{g}_{L^2_\diamond(\partial\Omega)}=1}
\sum_{n=N+1}^{+\infty} \frac{a_n^2+b_n^2}{n}\le \frac{1}{N+1}.
\]
This implies that the parameter $N$ (corresponding to $2N$ input currents) in order to have Lipschitz stability needs to satisfy
\[
N \ge c(\Omega,\lambda) C^2,
\]
and is therefore quadratic in the Lipschitz constant $C$. If $C$ depends exponentially on $\dim W$, this gives an exponential dependence of $N$ on $\dim W$. It is worth observing that for the Gel'fand-Calder\'on problem for the Schr\"odinger equation, one has a polynomial dependence if complex geometrical optics solutions (depending on the unknown) are used \cite{alberti2018}. 
\end{example}
Let us now prove Theorem~\ref{thm:eit}.

\begin{proof}
In the proof, the symbol $a\lesssim b$ will denote $a\le c\, b$, for some positive constant $c$ depending only on $\Omega$ and $\lambda$.

By using the variational formulation of \eqref{schr}, it is easy to see \cite{Lechleiter_2008} that $\ND \colon L^\infty_+(\Omega)\to  \mathcal{L}_c(L_\diamond^{2}(\partial \Omega),L_\diamond^{2}(\partial \Omega))$ is Fr\'echet differentiable and that its Fr\'echet derivative in $\sigma\in L^\infty_+(\Omega)$ in the direction $\tau\in L^\infty(\Omega)$ is given by
\[
\langle (\ND_\sigma'\tau)g,h\rangle_{L_\diamond^{2}(\partial \Omega)} = \int_\Omega \tau \nabla u^g_\sigma\cdot \nabla u^h_\sigma\,dx,\qquad g,h\in L_\diamond^{2}(\partial \Omega).
\]
Therefore, using the well-posedness of \eqref{schr}, namely  $\norm{\nabla u^g_\sigma}_{L^2(\Omega)}\lesssim \|J g\|_{H_\diamond^{-1/2}(\partial \Omega)}$,  for  $\tau\in L^\infty(\Omega)$ such that $\norm{\tau}_\infty\le 1$ we have
\begin{equation*}
\begin{split}
\norm{(\ND_\sigma'\tau)g}_{L_\diamond^{2}(\partial \Omega)}&=\sup_{\substack{h\in L_\diamond^{2}(\partial \Omega)\\ \|h\|_{L_\diamond^{2}(\partial \Omega)}=1}} |\langle (\ND_\sigma'\tau)g,h\rangle_{L_\diamond^{2}(\partial \Omega)} |\\
&=\sup_{\|h\|_{L_\diamond^{2}(\partial \Omega)}=1} \left| \int_\Omega \tau \nabla u^g_\sigma\cdot \nabla u^h_\sigma\,dx \right|\\
&\le  \sup_{\|h\|_{L_\diamond^{2}(\partial \Omega)}=1} \norm{\nabla u^g_\sigma}_{L^{2}(\Omega)} \norm{\nabla u^h_\sigma}_{L^{2}(\Omega)} \\
&\lesssim \|Jg\|_{H_\diamond^{-1/2}(\partial \Omega)}.
\end{split}
\end{equation*}
In particular, for $g\in L_\diamond^{2}(\partial \Omega)$ with $\norm{g}_{L_\diamond^{2}(\partial \Omega)}=1$ we have
\begin{equation}\label{eq:step1}
\norm{(\ND_\sigma'\tau)(I-P_N) g}_{L_\diamond^{2}(\partial \Omega)}\lesssim  \|J(I-P_N) g\|_{H_\diamond^{-1/2}(\partial \Omega)}\le \delta_N,
\end{equation}
where $\delta_N=\norm{J(I-P_N)}_{L^2_\diamond(\partial\Omega)\to H^{-\frac12}_\diamond(\partial\Omega)}$ and $I=I_{L^2_\diamond(\partial\Omega)}$.

Arguing in a similar way, using that $I-P_N$ is self-adjoint we obtain
\begin{equation}\label{eq:step2}
\norm{(I-P_N)(\ND_\sigma'\tau) g}_{L_\diamond^{2}(\partial \Omega)}\lesssim   \sup_{\|h\|_{L_\diamond^{2}(\partial \Omega)}=1}  \|J(I-P_N) h\|_{H_\diamond^{-1/2}(\partial \Omega)} =\delta_N.
\end{equation}

Recall the definition $s_N = \sup_{\xi \in K}\|R_N F'(\xi)\|_{W \to Y}$, where $R_N=I_Y-Q_N$. Arguing as in \eqref{eq:RN}, and using \eqref{eq:step1} and \eqref{eq:step2}, for $\xi\in K$ and $\tau\in W$ with $\norm{\tau}_{L^\infty(\Omega)}\le 1$ we have
\[
\begin{split}
\|R_N \ND_\xi'\tau\|_{Y}&=\sup_{\norm{g}_{L^2_\diamond(\partial\Omega)}=1} \|(R_N \ND_\xi'\tau)g\|_{L^2_\diamond(\partial\Omega)}   \\
&\le \sup_{\norm{g}_{L^2_\diamond(\partial\Omega)}=1} \|(I-P_N)(\ND_\xi'\tau)g\|_{L^2_\diamond(\partial\Omega)} + p \|(\ND_\xi'\tau)(I-P_N)g\|_{L_\diamond^{2}(\partial \Omega)}\\
&\lesssim p\,\delta_N.
\end{split}
\]
In particular, $s_N\lesssim p\,\delta_N$. 

Now recall that $\delta_N=\norm{J(I-P_N)}_{L^2_\diamond(\partial\Omega)\to H^{-\frac12}_\diamond(\partial\Omega)}$. By the Kondrachov embedding theorem, the operator $J$ is compact and so, by Lemma~\ref{lem:comp}, we have $\delta_N\to 0$.

Thus, the conclusion immediately follows from Theorem~\ref{theo:main}, part \ref{cond:sn}.
\end{proof}

\subsubsection*{Dirichlet-to-Neumann map}
The model based on the ND map is arguably more realistic in view of the applications \cite{eit-1999}. However, an alternative model for EIT considers boundary data modeled by the Dirichlet-to-Neumann (DN) map, namely  the operator 
\begin{equation*}
\Lambda_\sigma \colon H^{1/2}(\partial \Omega) \to H^{-1/2}(\partial \Omega),\qquad \Lambda_\sigma(f) = \sigma \left.\frac{\partial u^f_\sigma}{\partial \nu}\right|_{\partial \Omega},
\end{equation*}
where $u^f_\sigma\in H^1(\Omega)$ is the unique  solution of the Dirichlet problem for the conductivity equation
\begin{equation*}
\left\{
\begin{array}{ll}
-\nabla \cdot (\sigma \nabla u^f_\sigma) = 0 \quad &\text{in } \Omega,\\
u^f_\sigma =f \quad &\text{on } \partial \Omega.
\end{array}
\right.
\end{equation*}
Most  Lipschitz stability estimates for the inverse conductivity problem have been obtained for the DN map, i.e.\ they are of the form
\begin{equation*}
\|\sigma_1-\sigma_2\|_{L^\infty(\Omega)}\leq C \|\Lambda_{\sigma_1}-\Lambda_{\sigma_2}\|_{H^{1/2}(\partial \Omega) \to H^{-1/2}(\partial \Omega)},
\end{equation*}
for $\sigma_1,\sigma_2$ belonging to a known finite-dimensional subspace of $L^{\infty}(\Omega)$.\smallskip

Here we briefly sketch two possible approaches to extend Theorem~\ref{thm:eit} to this setting.

\begin{enumerate}
\item Given a Lipschitz stability estimate for the DN map, one can easily obtain the same stability (up to a slightly different constant) for the ND map. The only difference lies in the use of an integration by part (Alessadrini's identity) for the ND map instead of the usual one for the DN map. In this case, Theorem~\ref{thm:eit} applies directly.
\item It is also possible to avoid changing the boundary operator, and to obtain stability with a discretization of the DN map, as a consequence of Theorem~\ref{theo:main}. The compactness of the Frech\'et derivative, required to fit in the framework discussed in Example~\ref{ex:gal}, follows by the fact that $\Lambda_{\sigma_1}-\Lambda_{\sigma_2}$ is analytic smoothing \cite{mandache2001}, provided that $\sigma_1$ and $\sigma_2$ coincide in a neighborhood of the boundary.
\end{enumerate}

\subsection{Inverse scattering problem}

We now discuss the inverse medium problem in scattering theory \cite{2013-colton-kress}. The physical model here is 
\[
\left\{\begin{array}{ll}\Delta u+k^{2} n(x) u=0 \quad &\text { in } \mathbb{R}^{3}, \\ 
{u=u^{i}+u^{s}} &\text { in } \mathbb{R}^{3},
\end{array}\right.
\]
augmented with the Sommerfeld radiation condition
\[
\lim _{R \rightarrow+\infty} \int_{\partial B_{R}}\left|\frac{\partial u^{s}(y) }{ \partial r}-\mathrm{i} k u^{s}(y)\right|^{2}ds(y)=0,
\]
where $r=|y|$,  $k>0$ is the wavenumber, $n\in L^\infty(\R^3;\C)$ is the complex refractive index of the medium such that $\Imm(n)\ge 0$ in $\R^3$ and $\supp (1-n)\subseteq B$ for some open ball $B$ and $u^i$ is the incident field satisfying
\[
\Delta u^i+k^{2} u^i=0 \quad \text { in } \mathbb{R}^{3}.
\] 
The far-field, or scattering amplitude, is given by
\begin{equation*}
u_n^{\infty}(\hat{x})=\frac{1}{4 \pi} \int_{\partial B_{R}}\left(u^{s}(y) \frac{\partial e^{-i k \hat{x} \cdot y}}{\partial r}-\frac{\partial u^{s}}{\partial r}(y) e^{-i k \hat{x} \cdot y}\right) ds(y), \quad \hat{x} \in S^{2},
\end{equation*}
where $S^2$ is the unit sphere in $\R^3$ and $R>0$ is large enough so that $\overline B \subseteq B_R$. Choosing incoming waves $u^{i}(x)=e^{\mathrm{i} k x \cdot d}$ for $d\in S^2$, let $u^\infty_n(\cdot,d)$ denote the corresponding far-field measurements, so that $u_n^\infty\in L^2(S^2\times S^2)$. The following inverse  problem arises from this framework.

\vspace*{.2cm}

{\bf Fixed frequency inverse scattering problem.} Given $u^\infty_{n}\in L^2(S^2\times S^2)$ at fixed $k>0$, find $n$ in $B$.

\vspace*{.2cm}

As in the Calder\'on's problem, the stability of this inverse problem is only logarithmic in the general case \cite{1990-stefanov}, but Lipschitz estimates may be derived under a priori assumptions on $n$, and the theory developed in this paper may be applied.

We now set the various objects introduced in Section~\ref{sec:lip} (following \cite{bourgeois2013}). Let
\begin{itemize}
\item $X=L^\infty(B;\C)$;
\item $Y=L^2(S^2\times S^2)$;
\item $A=L^\infty_+(B)=\{f\in L^\infty(B;\C):\Imm (n)\ge \lambda\text{ in $B$ for some $\lambda>0$}\}$;
\item $F(n)=u_n^\infty$ be the far-field pattern associated to the refractive index $n$ extended by $1$ to the whole $\R^3$;
\item $W$ be a finite-dimensional subspace of $L^\infty(B;\C)$ and $K$ be a convex and compact subset of $W\cap A$;
\item and $Q_N\colon  L^2(S^2\times S^2)\to L^2(S^2\times S^2)$ be bounded linear maps, $N\in \N$, such that $Q_N=Q_N^*$ and $Q_N\to I_{L^2(S^2\times S^2)}$ strongly as $N\to+\infty$.
\end{itemize}
By using Theorem~\ref{theo:lip}, it was proven in \cite{bourgeois2013} that the inverse problem of recovering a refractive index $n$ in $K$ from its far-field pattern $u^\infty_n$ is Lipschitz stable, namely there exists $C>0$ such that
\[
\left\|n_{1}-n_{2}\right\|_{L^{\infty}(B)} \leqslant C\left\|u_{n_1}^{\infty}-u_{n_2}^{\infty}\right\|_{L^{2}\left(S^{2} \times S^{2}\right)},\qquad n_1,n_2 \in K.
\]
This estimate still requires the knowledge of the full measurements $u^\infty_{n_i}$ in $L^2(S^2\times S^2)$. By applying Theorem~\ref{theo:main} (or directly Corollary~\ref{cor:main}), we can establish the following result.
\begin{theorem}\label{thm:scattering}
Under the above assumptions, there exists $N\in\N$ (given explicitly by \eqref{cond:sn2}) such that
\begin{equation*}
\left\|n_{1}-n_{2}\right\|_{L^{\infty}(B)} \leqslant 2C\left\|Q_N(u_{n_1}^{\infty})-Q_N(u_{n_2}^{\infty})\right\|_{L^{2}\left(S^{2} \times S^{2}\right)},\qquad n_1,n_2 \in K.
\end{equation*}
\end{theorem}

Let us now show why $Q_N(u_{n_i}^{\infty})$ may be seen as finite measures.
Let the maps $Q_N$ be chosen as projections onto the vector spaces generated by the first elements of an ONB of $L^2(S^2\times S^2)$, as in Example~\ref{ex:hil}. More precisely, in this case we can consider the ONB of $L^2(S^2\times S^2)$ given by the tensor products of the spherical harmonics, namely
\[
\bigl\{ (\hat x,d) \mapsto Y^m_l(\hat x) \,Y^{m'}_{l'}(d): l,l'\in \N,\, |m|\le l,\, |m'|\le l'\bigr\},
\]
and let $Q_N$ be the projections onto
\[
\operatorname{span} \bigl\{ (\hat x,d) \mapsto Y^m_l(\hat x) \,Y^{m'}_{l'}(d): 0\le l,l'\le N,\, |m|\le l,\, |m'|\le l'\bigr\}.
\]
The measurements $Q_N(u^\infty)$ are now low-frequency projections of the full far-field pattern $u^\infty$. With this choice of the maps $Q_N$, it is also possible to apply the reconstruction algorithm discussed in Section~\ref{sec:rec}, which allows for the recovery of $n$ from the measurements $Q_N(u^\infty_n)$. 

Furthermore, let us show how to obtain  a Lipschitz stability estimate with only a finite number of directions $d$ and with  measurements of the corresponding far-field patterns taken only at a finite number of locations $\hat x$ on $S^2$. 

\begin{theorem}\label{cor:scattering}
Assume that the hypotheses of Theorem~\ref{thm:scattering} hold true. Let
\[
\{\hat x_i:i\in\N\}\subseteq S^2,\qquad \{d_j:j\in\N\}\subseteq S^2,
\]
be two dense subsets of $S^2$ (possibly identical). There exist $N\in\N$ and a constant $C'>0$  such that
\begin{equation*}
\left\|n_{1}-n_{2}\right\|_{L^{\infty}(B)} \leqslant  C'\left\|\bigl(u_{n_1}^{\infty}(\hat x_i,d_j)-u_{n_2}^{\infty}(\hat x_i,d_j)\bigr)_{i,j=1}^N\right\|_2,\qquad n_1,n_2 \in K.
\end{equation*}
\end{theorem}
\begin{proof}
We first observe that the far-field patterns are analytic functions of $\hat x$ and $d$ \cite{2011-kirsch}, and so taking pointwise evaluations is meaningful.

We consider the Sobolev space $H^3(S^2\times S^2)$ \cite{2013-brauchart-etal,2020-barcelo-etal,devito-2020,li2015sampling}. Since $\dim(S^2\times S^2)=4$ and $3>\frac{4}{2}$, $H^3(S^2\times S^2)$ is continuously embedded into $C(S^2\times S^2)$, and so it is a reproducing kernel Hilbert space consisting of continuous functions. As in  Example~\ref{ex:rkhs}, define
\[
G_N =\operatorname{span}\{k_{(\hat x_i,d_j)}:i,j=1,\dots,N\}\subseteq H^3(S^2\times S^2),
\]
and let $Q_N\colon H^3(S^2\times S^2)\to H^3(S^2\times S^2)$ be the projection onto $G_N$. Arguing as in the derivation of Theorem~\ref{thm:scattering} (with $Y=H^3(S^2\times S^2)$, since $u^\infty_{n_l}\in H^3(S^2\times S^2)$),
there exists $N\in\N$ such that
\[
\left\|n_{1}-n_{2}\right\|_{L^{\infty}(B)} \le 2C\left\|Q_N(u_{n_1}^{\infty})-Q_N(u_{n_2}^{\infty})\right\|_{H^{3}\left(S^{2} \times S^{2}\right)},\qquad n_1,n_2 \in K.
\]
Thus, in view of \eqref{eq:stable_sampling} we obtain
\[
\left\|n_{1}-n_{2}\right\|_{L^{\infty}(B)}
\le C'\left\|\bigl(u_{n_1}^{\infty}(\hat x_i,d_j)-u_{n_2}^{\infty}(\hat x_i,d_j)\bigr)_{i,j=1}^N\right\|_2,\qquad n_1,n_2 \in K,
\]
for some $C'>0$, as desired.
\end{proof}

\subsection{Quantitative photoacoustic tomography (QPAT)}
Let $\Omega\subseteq\R^d$ be a bounded domain of class $C^{1,\alpha}$ for some $\alpha\in (0,1)$ and $\mu\in L^\infty(\Omega)$ be a nonnegative function representing the optical absorption of the medium. Photoacoustic tomography (PAT) is a hybrid modality based on coupling optical and ultrasonic waves \cite{Wang1458}: a laser pulse illuminates the biological tissue under consideration, whose thermal expansion creates a ultrasonic wave that can be measured outside the medium.
In the first step of photoacoustic tomography \cite{2015-kuchment-kunyansky,alberti-capdeboscq-2018}, by solving a linear (and stable) inverse source problem for the wave equation, we measure the internal optical energy
\[
\mu(x) u(x),\qquad \text{a.e. }x\in \Omega,
\]
where the light intensity $u\in H^1(\Omega)$ solves
\begin{equation}\label{eq:light}
\left\{\begin{array}{ll}
-\Delta u+\mu u=0 &\text{in $\Omega$,}\\
u=\varphi &\text{on $\partial\Omega$,}
\end{array}
\right.
\end{equation} 
for a fixed known illumination $\varphi\in C^{1,\alpha}(\overline\Omega)$ such that $\min\varphi>0$.

We have considered the diffusion approximation for light propagation with constant diffusion term. More involved models, involving a non-constant leading order term \cite{BAL-UHLMANN-2010,BAL-REN-2011} or the use of the more accurate transport equation  \cite{2010-bal-jollivet-jugnon}, would complicate the analysis, but the type of results obtained would be similar: we have decided to discuss the simplest model in order to highlight the impact of the results of this paper.

The second step of PAT, called quantitative, consists in the reconstruction of $\mu$ from the knowledge of the internal energy.

\vspace*{.2cm}

{\bf Inverse problem of QPAT.} Given $\mu u$ in $\Omega$, find $\mu$ in $\Omega$.

\vspace*{.2cm}

We now set the various objects introduced in Section~\ref{sec:lip}. Let
\begin{itemize}
\item $X=W$ be a finite-dimensional subspace of $L^\infty(\Omega)$;
\item $Y=L^2(\Omega)$;
\item $A=\{f\in X:\lambda^{-1}<\mu< \lambda\text{ in $\Omega$ for some $\lambda>0$}\}$;
\item $F(\mu)=\mu u$, where $u\in H^1(\Omega)$ is the unique weak solution of \eqref{eq:light};
\item $K=\{f\in X:\Lambda^{-1}\le\mu\le \Lambda\text{ in $\Omega$}\}$ for some fixed $\Lambda>0$;
\item and $Q_N\colon  L^2(\Omega)\to L^2(\Omega)$ be bounded linear maps, $N\in \N$, such that $Q_N=Q_N^*$ and $Q_N\to I_{L^2(\Omega)}$ strongly as $N\to+\infty$.
\end{itemize}

Several Lipschitz stability estimates have been derived for PAT \cite{BAL-UHLMANN-2010,BAL-REN-2011,KS-IP-12}. For completeness, we provide a proof in the setting considered here.
\begin{proposition}\label{prop:pat}
There exists a constant $C>0$ depending only on $\Omega$, $W$, $\Lambda$ and $\varphi$ such that
\begin{equation*}
\|\mu_1-\mu_2\|_\infty\le C \|\mu_1 u_1- \mu_2 u_2\|_2,\qquad \mu_1,\mu_2\in K,
\end{equation*}
where $u_i\in H^1(\Omega)$ is the solution to \eqref{eq:light} with coefficient $\mu=\mu_i$.
\end{proposition}

\begin{remark*}
As it will be clear from the proof, the dependence of $C$ on $W$ appears only to replace the $L^2$ norm of $\mu_1-\mu_2$ on the left hand side by the more natural $L^\infty$ norm. This dependence may be dropped if higher order Sobolev (or H\"older) norms are used \cite{BAL-UHLMANN-2010,BAL-REN-2011}.
\end{remark*}

\begin{proof}
Let us denote the unique solution to \eqref{eq:light} by $u(\mu)$, so that $F(\mu)=\mu\, u(\mu)$. In the proof, by an abuse of notation, several different positive constants depending only on $\Omega$, $W$, $\Lambda$ and $\varphi$ will be denoted by the same letter $C$.

We claim that there exists $C>0$ such that
\begin{equation}\label{eq:claim}
u(\mu) \ge C\quad\text{in}\;\Omega,\qquad \mu \in K.
\end{equation}
Take $\mu \in K$. By classical elliptic regularity theory (see, e.g., \cite[Theorem~8.29]{GILBARG-2001}) we have that $u(\mu)\in C^1(\overline\Omega)$ and
\begin{equation}\label{eq:regularity}
\|u(\mu)\|_{C^1(\overline\Omega)} \le C.
\end{equation}
As a consequence, since $\varphi$ is positive, there exists $\Omega'\Subset\Omega$ such that
\begin{equation}\label{eq:nearboundary}
u(\mu)\ge \frac{\min\varphi}{2}\quad\text{in}\;\Omega\setminus\Omega'.
\end{equation}
By the strong maximum principle (see, e.g., \cite[Theorem~8.19]{GILBARG-2001} applied to $-u(\mu)$) we have $u(\mu)>0$ in $\Omega$. Thus, the Harnack inequality (\cite[Theorem~8.20]{GILBARG-2001}) yields
\begin{equation}\label{eq:interior}
u(\mu)\ge C\quad\text{in}\;\Omega'.
\end{equation}
Finally, combining \eqref{eq:nearboundary} and \eqref{eq:interior} we obtain \eqref{eq:claim}.

In view of \eqref{eq:light} we have that
\begin{equation*}
\left\{\begin{array}{ll}
-\Delta (u(\mu_1)-u(\mu_2))=F(\mu_2)-F(\mu_1) &\text{in $\Omega$,}\\
u(\mu_1)-u(\mu_2)=0 &\text{on $\partial\Omega$.}
\end{array}
\right.
\end{equation*}
Thus, standard energy estimates for the Poisson equation give
\begin{equation}\label{eq:poisson}
\|u(\mu_1)-u(\mu_2)\|_2\le C \|F(\mu_1)-F(\mu_2)\|_2,\qquad \mu_1,\mu_2\in K.
\end{equation}

Let us now show the Lipschitz stability estimate of the statement. Using the identity
\[
\mu_1-\mu_2=\frac{F(\mu_1)-F(\mu_2)}{u(\mu_1)} +F(\mu_2)\frac{u(\mu_2)-u(\mu_1)}{u(\mu_1)u(\mu_2)},\qquad\mu_1,\mu_2\in K,
\]
thanks to \eqref{eq:claim}, \eqref{eq:regularity} and \eqref{eq:poisson} we readily derive for $\mu_1,\mu_2\in K$
\begin{equation*}
\begin{split}
\norm{\mu_1-\mu_2}_2&\le \frac{\|F(\mu_1)-F(\mu_2)\|_2}{\inf u(\mu_1)} +\norm{F(\mu_2)}_\infty \frac{\|u(\mu_1)-u(\mu_2)\|_2}{\inf u(\mu_1)\,\inf u(\mu_2)}\\
&\le C \|F(\mu_1)-F(\mu_2)\|_2.
\end{split}
\end{equation*}
Finally, the $L^2$ norm of $\mu_1-\mu_2$ may be replaced by the $L^\infty$ norm because the space $W$ is finite dimensional and all norms are equivalent.
\end{proof}

Note that the map $F$ is certainly Fr\'echet differentiable because it is in fact analytic (the map $\mu\in A\mapsto u(\mu)\in Y$ is analytic). Therefore, it is possible to apply Theorem~\ref{theo:main} to this inverse problem and obtain a Lipschitz stability estimate with finite-dimensional measurements.

\begin{theorem}
Under the above assumptions, there exists $N\in\N$ (given explicitly by \eqref{cond:sn2}) such that
\begin{equation}
\|\mu_1-\mu_2\|_\infty\le 2C \|Q_N(\mu_1 u_1)-Q_N(\mu_2 u_2)\|_2,\qquad \mu_1,\mu_2\in K,
\end{equation}
where $u_i\in H^1(\Omega)$ is the solution to \eqref{eq:light} with coefficient $\mu=\mu_i$.
\end{theorem}

As in the previous examples, the projections $Q_N$ may be chosen as low-pass filters. The reconstruction algorithm discussed in Section~\ref{sec:rec} may be used to recover the unknown $\mu$ from finite measurements.

We considered PAT as a key example, but the analysis presented here may be easily generalized to several other hybrid imaging inverse problems with internal data \cite{B-IO-12,alberti-capdeboscq-2018}, such as thermo-acoustic tomography, dynamic elastography and acousto-electric tomography.

\section{Conclusion}\label{sec:conclusions}
In this paper, we have discussed a general framework to derive Lipschitz stability estimates for nonlinear inverse problems with finite measurements, under the assumption that the unknown belongs to a finite-dimensional space. A global reconstruction algorithm was also derived, based on a nonlinear Landweber iteration. We then applied the general theory to Calder\'on's inverse problem for EIT, to inverse scattering and to quantitative photoacoustic tomography.

Let us discuss a few research directions motivated by the findings of this work.
\begin{itemize}
\item The assumption of the finite dimensionality of $W$ is not needed in Theorem~\ref{theo:main}, part \ref{cond:sn}. Thus, it would be interesting to investigate whether the condition $s_N\le \frac{1}{2C}$ may be derived in some cases with the only compactness assumption on $K$. The latter could follow from a more general argument based on compact embeddings.
\item The theory of compressed sensing (CS) allows for the recovery of sparse signals from a number of (linear) measurements that is proportional to the sparsity (up to log factors), under suitable incoherence assumptions. The classical theory, working for a unitary forward map  $F$, was recently extended to arbitrary linear maps with bounded inverse \cite{alberti2017infinite}. It would be interesting to apply CS to the general nonlinear setup presented in this paper by using the results of \cite{grasmair-etal-2011,grasmair-etal-2011b}.
\item The numerical implementation of the reconstruction algorithm constructed in Section~\ref{sec:rec} would allow for testing its efficiency and applicability.
\item In this work, we have limited ourselves to considering three nonlinear inverse problems as proofs of concept, but it would be interesting to apply the theory to other examples,  for instance where the unknown is supposed to have a particular shape (e.g.\ a polygon) with unknown location \cite{beretta2015,beretta2019} or for inverse problems with fractional operators \cite{ruland2018,cekic2018calder}.
\item In Theorem~\ref{theo:main}, the dependence on the mismodeling error in the stability estimate is not sharp, as discussed in Remark~\ref{rem:2}. A very challenging question is then to improve upon this dependence, without losing the generality of our approach.
\end{itemize}

\bibliographystyle{abbrvurl}
\bibliography{CS}

\appendix

\section{Local convergence of Landweber iteration}\label{sec:app}
For the sake of completeness, we describe the convergence result used in Section~\ref{sub:local}. We present a simplified version of  \cite[Theorem~3.2]{dehoop2012} that is sufficient for our scopes.

 Let $X$ and $Y$ be Hilbert spaces, $A\subseteq X$ be an open set, $K\subseteq A$ be a compact set and $F\colon A\to Y$ be such that
\begin{enumerate}
\item $F\in C^1(A,Y)$ and $F'\colon A\to \mathcal{L}_c(X,Y)$ is Lipschitz continuous, namely
\[
\|F'(x_1)-F'(x_2)\|_{X\to Y}\le L\|x_1-x_2\|_X,\qquad x_1,x_2\in A,
\]
for some $L>0$;
\item $F^{-1}$ is Lipschitz continuous, namely
\[
\|x_1-x_2\|_X \le C\|F(x_1)-F(x_2)\|_Y,\qquad x_1,x_2\in A,
\]
for some $C>0$.
\end{enumerate}
\begin{proposition}\label{prop:land}
There exist $\rho,c,\mu\in (0,1)$ such that the following is true. Take $x^\dagger\in K$ and let $y=F(x^\dagger)$. If $x_0\in K$ satisfies
\[
\|x^\dagger-x_0\|<\rho,
\]
then the iterates $(x_k)$ of the Landweber iteration
\[
x_{k+1}=x_{k}-\mu  F'(x_{k})^{*}\left(F(x_{k})-y\right),\qquad k\in \N,
\]
converge to $x^\dagger$ and satisfy
\[
\|x^\dagger-x_k\|\le \rho c^k,\qquad k\in \N.
\]
\end{proposition}

\begin{remark*}
The constants $\rho$, $c$ and $\mu$ are given explicitly in \cite{dehoop2012} as functions of the a priori data. The statement provided  in \cite{dehoop2012} requires $F$ to be weakly sequentially closed, but a close look to the proof indicates that this assumption may be dropped, since in this case the existence of the minimizer $x^\dagger$ is granted.
\end{remark*}

\end{document}